\documentclass{amsart}
\usepackage{graphicx}
\usepackage{color}
\newtheorem{thm}{Theorem}[section]
\theoremstyle {definition}
\newtheorem{cor}[thm]{Corollary}
\newtheorem{prop}[thm]{Proposition}
\newtheorem{defn}[thm]{Definition}
\newtheorem{lem}[thm]{Lemma}
\newtheorem{eg}[thm]{Example}

\newtheorem{rem}[thm]{Remark}

\newtheorem{que}[thm]{Question}
\numberwithin{equation}{section}

\begin{document}

\title[The Zariski topology-graph on the maximal spectrum]
{The Zariski topology-graph on the maximal spectrum of modules over commutative rings}%
\author{Habibollah Ansari-Toroghy }%
\address{Department of pure Mathematics, Faculty of mathematical Sciences, University of
Guilan,
P. O. Box 41335-19141 Rasht, Iran.} %
\email{ansari@guilan.ac.ir}%

\author{Shokoofeh Habibi}%
\address{Department of pure Mathematics, Faculty of mathematical Sciences, University of
Guilan,
P. O. Box 41335-19141 Rasht, Iran.}%
\email{sh.habibi@phd.guilan.ac.ir}%


\date{\today}%
\begin{abstract}
Let $M$ be a module over a commutative ring and let $Spec(M)$
(resp. $Max(M)$) be the collection of all prime (resp. maximal)
submodules of $M$. We topologize $Spec(M)$ with Zariski topology,
which is analogous to that for $Spec(R)$, and consider $Max(M)$ as
the induced subspace topology. For any non-empty subset $T$ of
$Max(M)$, we introduce a new graph $G(\tau^{m}_{T})$,
 called the Zariski topology-graph on the maximal spectrum of $M$. This graph helps us to
  study the algebraic (resp. topological) properties of $M$ (resp. $Max(M)$) by using the
 graph theoretical tools.
\end{abstract}
\maketitle
\section{Introduction}
Throughout this paper $R$ is a commutative ring with a non-zero
identity and $M$ is a unital $R$-module. By $N\leq M$ (resp. $N<
M$) we mean that $N$ is a submodule (resp. proper submodule) of
$M$.

There are many papers on assigning graphs to rings or modules.
Annihilating-ideal graphs of rings, first introduced and studied
in \cite{br11}, provide an excellent setting for studying the
ideal structure of a ring. $AG(R)$, the Annihilating-ideal graph
of $R$, to be a graph whose vertices are ideals of $R$ with
non-zero annihilators and in which two vertices $I$ and $J$ are
adjacent if and only if $IJ=0$.

In \cite{ah}, we generalized the above idea to submodules of $M$
and define the (undirected) graph $AG(M)$, called \textit {the
annihilating submodule graph}, with vertices\\ $V(AG(M))$= $\{N
\leq M:$ there exists $\{0\}\neq K<M$ with $NK=0$ \}, where $NK$,
the product of $N$ and $K$, is defined by $(N:M)(K:M)M$ (see
\cite{af07}). In this graph, distinct vertices $N,L \in AG(M)$ are
adjacent if and only if $NL=0$. In section two of this article, we
collect some fundamental properties of the annihilating-submodule
graph of a module which will be used in this work.

A prime submodule (or a $p$-prime submodule) of $M$ is a proper
submodule $P$ of $M$ such that whenever $re\in P$ for some $r\in
R$ and $e \in M$, either $e\in P$ or $r\in p$ \cite{lu84}.

The prime spectrum (or simply, the spectrum) of $M$ is the set of
all prime submodules of $M$ and denoted by $Spec(M).$

The \textit{Zariski topology} on $Spec(M)$ is the topology
$\tau_M$ described by taking the set $Z(M)= \{ V(N): N \leq M \}$
as the set of closed sets of $Spec(M)$, where $V(N)=\{P \in
Spec(M): (P:M)\supseteq (N:M) \}$ \cite{lu99}.

The closed subset $V(N)$, where $N$ is a submodule of $M$, plays
an important role in the Zariski topology on $Spec(M)$. In
\cite{ah}, We employed these sets and defined a new graph \textit
{$G(\tau_T)$}, called the \textit{Zariski topology-graph}. This
graph helps us to study algebraic (resp.topological) properties of
modules (resp. Spec(M)) by using the graphs theoretical tools.
$G(\tau_T)$ is an undirected graph with vertices $V(G(\tau_T)$=
$\{N < M :$ there exists $K < M$ such that $V(N)\cup V(K)=T$ and
$V(N),V(K)\neq T\}$, where $T$ is a non-empty subset of $Spec(M)$
and distinct vertices $N$ and $L$ are adjacent if and only if
$V(N)\cup V(L)=T$.

There exists a topology on $Max(M)$ having $Z^{m}(M) = \{
V^{m}(N): N \leq M \}$ as the set of closed sets of $Max(M)$,
where $V^{m}(N) = \{Q\in Max(M):(Q : M) \supseteq (N : M)\}$. We
denote this topology by $\tau^{m}_M$. In fact $\tau^{m}_M$ is the
same as the subspace topology induced by $\tau_M$ on $Max(M)$.

In this paper, we define a new graph $G(\tau^{m}_T)$, called the
\textit{Zariski topology-graph on the maximal spectrum of M},
where $T$ is a non-empty subset of $Max(M)$, and by using this
graph, we study algebraic (resp. topological) properties of $M$
(resp. $Max(M)$). $G(\tau^{m}_T)$ is an undirected graph with
vertices $V(G(\tau^{m}_T))$= $\{N < M :$ there exists a non-zero
proper submodule $L$ of $M$ such that $V^{m}(N)\bigcup V^{m}(L)=T$
and $V^{m}(N),V^{m}(L)\neq T\}$, where $T$ is a non-empty subset
of $Max(M)$ and distinct vertices $N$ and $L$ are adjacent if and
only if $V^{m}(N)\bigcup V^{m}(L)=T$.

Let $T$ be a non-empty subset of $Max(M)$. As $\tau^{m}_M$ is the
subspace topology induced by $\tau_M$ on $Max(M)$, one may think
that $G(\tau_T)$ and $G(\tau^{m}_T)$ have the identical nature.
But the Example \ref{e3.5} (case (1)) shows that this is not true
and these graphs are different. Also the Example \ref{e3.5} (case
(2)) shows that for a non-empty subset $T'$ of $Spec(M)$, under a
condition, $G(\tau_{T'})$ can be regarded as a subgraph of
 $G(\tau^{m}_T)$, where $T= T'\cap Max(M)$. Besides, this case denotes
 that $G(\tau_{T'})$ is not a subgraph of $G(\tau^{m}_T)$,
 necessarily. Moreover, it is shown that $G(\tau^{m}_T)$ can not appear as a
subgraph of $G(\tau_{T'})$, where $T\subseteq T'\subseteq
Spec(M)$, in general (see Example \ref{e3.5} (case (3)). So, the
results related to $G(\tau^{m}_T)$, where $T$ is a non-empty
subset of $Max(M)$, do not go parallel to those of $G(\tau_{T'})$,
where $T'$ is a non-empty subset of $Spec(M)$, necessarily. Based
on the above remarks, it is worth to study Max-graphs separately.

For any pair of submodules $N\subseteq L$ of $M$ and any element
$m$ of $M$, we denote $L/N$ and the
  residue class of $m$ modulo $N$ in $M/N$ by
  $\overline{L}$ and $\overline{m}$ respectively.

 For a submodule $N$ of $M$, the \textit{colon ideal of $M$ into $N$} is defined by
 $(N:M)=\{r\in R: rM\subseteq N\}=Ann(M/N)$.
  Further if $I$ is
  an ideal of $R$, the submodule $(N:_MI)$ is defined by
 $\{ m \in M: Im \subseteq N \}$. Moreover, $\Bbb Z$ (resp. $\Bbb Q$) denotes the ring of integers
  (resp. the field of rational numbers).

The prime radical $\sqrt{N}$ is defined to be the intersection of
all prime submodules of $M$ containing $N$, and in case $N$ is not
contained in any prime submodule, $\sqrt{N}$ is defined to be $M$.
Note that the intersection of all prime submodule $M$ is denoted
by $rad(M)$.

If $Max(M)\neq\emptyset$, the mapping $\psi :Max(M)\rightarrow
Max(R/Ann(M))= Max(\overline{R})$ such that $\psi
(Q)=(Q:M)/Ann(M)=\overline{(Q:M)}$ for every  $Q \in Max(M)$, is
called the \textit{natural map} of $Max(M)$ \cite{HR103}.

$M$ is said to be \textit{Max-surjective} if either
$M=\textbf{(0)}$ or $M\neq \textbf{(0)}$ and the natural map of
$Max(M)$ is surjective \cite{HR103}.

For a proper ideal $I$ of $R$, we recall that the $J-radical$ $I$,
denoted by $J^{m}(I)$, is the intersection of all maximal ideals
containing $I$.\\
The $J-radical$ of a submodule $N$ of $M$, denoted by $J^{m}(N)$,
is the intersection of all members of $V^{m}(N)$. In case that
$V^{m}(N)= \emptyset$, we define $J^{m}(N) = M$ \cite{ak}.

 A topological space $X$ is said to be connected if there
doesn't exist a pair $U$, $V$ of disjoint non-empty open sets of
$X$ whose union is $X$. A topological space $X$ is irreducible if
for any decomposition $X= X_1 \bigcup X_2$ with closed subsets
$X_i$ of $X$ with $i= 1, 2$, we have $X= X_1$ or $X= X_2$.
 A subset $X'$ of $X$ is connected (resp. irreducible) if it is connected
 (resp. irreducible) as a subspace of $X$.

In section 2, we briefly review some fundamental properties of the
annihilating-submodule graph of a module needed later.

In section 3, among other results, it is shown that the
\textit{Zariski topology-graph on maximal Spectrum of M} ,
$G(\tau^{m}_T)$, is connected and $diam(G(\tau^{m}_T))\leq 3$.
Moreover, if $G(\tau^{m}_T)$ containing a cycle satisfies
$gr(G(\tau^{m}_T))\leq 4$ (see Theorem \ref{t3.6}). Also we
consider some conditions under which $G(\tau^{m}_T)$ is a
non-empty graph.

In section 4, the relationship between $G(\tau^{m}_T)$ and $AG(M/
\Im(T))$ is investigated, where $\Im(T)$ is the intersection of
all members of $T$. It is proved that if $M$ is a Max-surjective
$R$-module and $N$ and $L$ are proper submodules of $M$ which are
adjacent in $G(\tau^{m}_T)$, then $J^{m}(N)/\Im(T)$ and
$J^{m}(L)/\Im(T)$ are adjacent in $AG(M/ \Im(T))$ (see Corollary
\ref{c4.2}). Also it is shown, under some conditions, that $AG(M/
\Im(T))$ is isomorphic to a subgraph of $G(\tau^{m}_T)$ and $AG(M/
\Im(T))$ is non-empty if and only if $G(\tau^{m}_T)$ is non-empty,
and any two proper submodules $N$ and $L$ of $M$ are adjacent in
$G(\tau^{m}_T)$ if $N/ \Im(T)$ and $L/ \Im(T)$ are adjacent in
$M/\Im(T)$ (see Proposition \ref{p4.3}).

For completeness, we now mention some graph theoretic notions and
notations that are used in this article. A graph $G$ is an ordered
triple $(V(G), E(G), \psi_G )$ consisting of a non-empty set of
vertices, $V(G)$, a set $E(G)$ of edges, and an incident function
$\psi_G$ that associates an unordered pair of distinct vertices
with each edge. The edge $e$ joins $x$ and $y$ if $\psi_G(e)=\{x,
y\}$, and we say $x$ and $y$ are adjacent. The degree $d_G(x)$ of
a vertex $x$ is the number of edges incident with $x$. A path in
graph $G$ is a finite sequence of vertices $\{x_0, x_1,\ldots
,x_n\}$, where $x_{i-1}$ and $x_i$ are adjacent for each $1\leq
i\leq n$ and we denote $x_{i-1} - x_i$ for existing an edge
between $x_{i-1}$ and $x_i$. The number of edges crossed to get
from $x$ to $y$ in a path is called the length of the path, where
$x, y\in V(G)$. A graph $G$ is connected if a path exists between
any two distinct vertices. For distinct vertices $x$ and $y$ of
$G$, let $d(x,y)$ be the length of the shortest path from $x$ to
$y$ and if there is no such path $d(x,y)=\infty$. The diameter of
$G$ is $diam(G)=sup\{ d(x, y) : x,y\in V(G)\}$. The girth of $G$,
denoted by $gr(G)$, is the length of a shortest cycle in $G$
($gr(G)=\infty $ if $G$ contains no cycle)(see \cite{al99}).

A graph $H$ is a subgraph of $G$ if $V(H)\subseteq V(G)$,
$E(H)\subseteq E(G)$ and $\psi_H$ is the restriction of $\psi_G$
to $E(H)$. Two graphs $G$ and $G'$ are said
 to be isomorphic if there exists a one-to-one and onto function $\phi: V(G)\rightarrow V(H)$ such that
  if $x, y\in V(G)$, then $x - y$ if and only if $\phi(x) -  \phi(y)$.
A bipartite graph is a graph whose vertices can be divided into two
disjoint sets $U$ and $V$ such that every edge connects a vertex in
$U$ to one in $V$; that is, $U$ and $V$ are each independent sets and complete bipartite graph on $n$
 and $m$ vertices, denoted by $K_{n, m}$, where
$V$ and $U$ are of size $n$ and $m$ respectively, and $E(G)$ connects every vertex in $V$ with all vertices in
$U$ (see \cite{r05}).

In the rest of this article, $T$ denotes a non empty subset of
$Max(M)$ and $\Im(T))$ is the intersection of all members of $T$.

\section {Previous results}
As we mentioned before, $AG(M)$ is a graph with vertices $V(AG(M))
=\{ N \leq M: NL=0$ for some $\{0\} \neq L < M\}$, where distinct
vertices $N$ and $L$ are adjacent if and only if $NL=0$ (here we
recall that the product of $N$ and $L$ is defined by
$(N:M)(L:M)M$) (see \cite[Definition 3.1]{ah}).\\

The following results reflect some basic properties of the
annihilating-submodule graph of a module.\\

\textbf{Proposition A} ([4, Proposition 3.2]). Let $N$ be a
non-zero proper submodule of $M$.
\begin{itemize}
\item[(a)] $N$ is a vertex in $AG(M)$ if $Ann(N)\neq Ann(M)$ or
$(0:_{M}(N:M))\neq 0$. \item[(b)] $N$ is a vertex in $AG(M)$,
where $M$ is a multiplication module, if and only if
$(0:_{M}(N:M))\neq 0$.\\
\end{itemize}
\begin{rem} \label{r3.1}
In the annihilating-submodule graph $AG(M)$, $M$ itself can be a
vertex. In fact $M$ is a vertex if and only if every non-zero
submodule is a vertex if and only if there exists a non-zero
proper submodule $N$ of $M$ such that $(N:M)=Ann(M)$. For example,
for every submodule $N$ of $\Bbb Q$ (as $\Bbb Z$-module), $(N:\Bbb
Q)=0$. Hence $\Bbb Q$ is a vertex in $AG(\Bbb Q)$.\\
\end{rem}

\textbf{Theorem B} ([4, Theorem 3.3]). Assume that $M$ is not a
vertex. Then the following hold.
\begin {itemize}
\item [(a)] $AG(M)$ is empty if and only if $M$ is a prime module.
\item [(b)] A non-zero submodule $N$ of $M$ is a vertex if and
only if $(0:_{M}(N:M))\neq 0$.\\
\end {itemize}

\textbf{Theorem C} ([4, Theorem 3.4]). The annihilating-submodule
graph $AG(M)$ is connected and $diam(AG(M))\leq 3$. Moreover, if
$AG(M)$ contains a cycle, then $gr(AG(M))\leq 4$.\\

\textbf{Proposition D} ([4, Proposition 3.5]). Let $R$ be an
Artinian ring and $M$ a finitely generated $R$-module. Then every
nonzero proper submodule $N$ of $M$ is a vertex in $AG(M)$.\\

\textbf{Theorem E} ([4, Theorem 3.6]). Suppose $M$ is not a prime
module. Then $AG(M)$ has acc (resp. dcc) on vertices if and only
if $M$ is a Noetherian (resp. an Artinian) module.\\

\textbf{Theorem F} ([4, Theorem 3.7]). Let $R$ be a reduced ring
and let $M$ be a faithful module which is not prime. Then the
following statements are equivalent.
\begin {itemize}
\item [(a)] $AG(M)$ is a finite graph. \item [(b)] $M$ has only
finitely many submodules. \item [(c)] Every vertex of $AG(M)$ has
finite degree. Moreover, $AG(M)$ has n ($n\geq 1$) vertices if and
only if $M$ has only $n$ nonzero proper submodules.\\
\end {itemize}

\textbf{Proposition G} ([4, Proposition 3.9]). We have exactly one
of the following assertions.
\begin {itemize}
\item [(a)] Every non-zero submodule $M$ is a vertex in $AG(M)$.
\item [(b)] There exists maximal ideal $m$ of $R$ such that $mM\in
V(AG(M))$ if and only if $Soc(M)\neq 0$.
\end {itemize}

\section{The Zariski topology-graph on the maximal spectrum module}

\begin{defn}\label{d3.1} We define $G(\tau^{m}_{T})$, a
\textit{Zariski topology-graph on the maximal spectrum of M}, with
vertices $V(G(\tau^{m}_{T}))$= $\{N < M :$ there exists $L < M$
such that $V^{m}(N)\bigcup V^{m}(L)=T$ and $V^{m}(N),V^{m}(L)\neq
T\}$, where
 $T$ is a non-empty subset of $Max(M)$ and distinct vertices $N$ and $L$ are
 adjacent if and only if $V^{m}(N)\bigcup V^{m}(L)=T$.
\end{defn}

\begin{lem}\label{l3.2} $G(\tau^{m}_{T})\neq \emptyset$ if and only if $T$ is
 closed and is not irreducible subset of $Max(M)$.
\end{lem}

\begin{proof}
Let $N$ and $L$ be submodules of $M$. Since $V^{m}(N)=Max(M)\cup
V(N)$, by \cite [Result 3]{lu99}, we have $V^{m}(N)\cap
V^{m}(L)=V^{m}(N\cup L)=V^{m}(NL)$ $(=V^{*m}(NL)$. Note that
$V^{*m}(N)= \{P \in Max(M): P\supseteq N \}$. Then for every
submodules $N$ and $L$ of $M$, we have $V^{*m}(N)\cup
V^{*m}(L)\subseteq V^{*m}(N \cup L))$. Hence, the proof is
Straightforward.
\end{proof}

\begin{rem}\label{r3.3} By \cite [Lemma 3.5]{ak}, $T$ is closed if and only if
 $T=V^{m}(\Im(T))$. Hence $G(\tau^{m}_{T})\neq \emptyset $ if and only if
$T=V^{m}(\Im(T))$ and $T$ is not irreducible.
\end{rem}

\begin{rem}\label{r3.4} By \cite[Theorem 3.13(a)]{ak}, if $M$ is a Max-surjective
$R$-module, then $G(\tau^{m}_{T})\neq \emptyset$ if and
only if $T=V^{m}(\Im(T))$ and $(\Im(T):M)$ is
 not a $J$-radical prime ideal of $R$. If $Spec(M)=Max(M)$ and $G(\tau^{m}_T)\neq \emptyset$,
 then $T=V^{m}(\Im(T))$ and $\Im(T)$ is not a prime submodule of $M$ by \cite[Proposition 5.4]{lu99}.
\end{rem}

\begin{eg}\label{e3.5} Consider the following examples.

case (1): Set $R:=\Bbb Z\bigoplus \Bbb Z$. Then $Max(R)=\{\Bbb
Z\bigoplus p_{i}\Bbb Z,  p_{i}\Bbb Z \bigoplus \Bbb Z: i\in \Bbb
N\}$ and $Spec(R)=Max(R)\cup \{(\textbf{0})\bigoplus \Bbb Z, \Bbb
Z\bigoplus (\textbf{0})\}$.\\
In this example, we see that $G(\tau_{Max(R)})$ and
$G(\tau^{m}_{Max(R)})$ are different. Because
$I=(\textbf{0})\bigoplus \Bbb Z$ and $J=\Bbb Z\bigoplus
(\textbf{0})$ are not adjacent in $G(\tau_{Max(R)})$ but they are
adjacent in $G(\tau^{m}_{Max(R)})$. In fact, $G(\tau_{Spec(R)})$
and $G(\tau^{m}_{Max(R)})$ are complete bipartite graphs with two
parts $U=\{ I\bigoplus (\textbf{0})\}$ and
$V=\{(\textbf{0})\bigoplus J\}$, where $I$ and $J$ are nonzero
proper ideals of $\Bbb Z$.

case (2): Set $R:=\Bbb Q\bigoplus \Bbb Z$. Then $Max(R)=\{\Bbb
Q\bigoplus p_{i}\Bbb Z: i\in \Bbb N\}$ and $Spec(R)=Max(R)\cup
\{(\textbf{0})\bigoplus \Bbb Z, \Bbb Q\bigoplus (\textbf{0})\}$.\\
In this example, we see that $G(\tau_{Spec(R)})$ is not a subgraph
of $G(\tau^{m}_{Max(R)})$. Because $G(\tau_{Spec(R)})$ is a
complete bipartite graph with two parts $U=\{ I\bigoplus
(\textbf{0})\}$ and $V=\{(\textbf{0})\bigoplus J\}$, where $I$ and
$J$ are non-zero proper ideals of $\Bbb Q$ and $\Bbb Z$,
respectively and $G(\tau^{m}_{Max(R)})$ is an empty graph. In
fact, for every non-empty subset $T$ of $Spec(M)$, $G(\tau_T)$ is
a subgraph of $G(\tau^{m}_{T'})$ (i.e., $V(G(\tau_T))\subseteq
V(G(\tau^{m}_{T'}))$ and $E(G(\tau_T))\subseteq
E(G(\tau^{m}_{T'}))$) iff for every vertex $N$ of $G(\tau_T)$,
$V^{m}(N)\neq T'$, where $T'=T\cap Max(M)$.

case (3): Set $R:=\Bbb Q\bigoplus \Pi_{i\in \Bbb N} \Bbb
Z/p_{i}Z$. Then $Max(R)=\{p_{i}R, i\in \Bbb N\}$ and
$Spec(R)=Max(R)\cup
\{(\textbf{0})\bigoplus \Pi_{i\in \Bbb N} \Bbb Z/p_{i}Z\}$.\\
For every $i, j\in \Bbb N$, $i\neq j$,  $I_{i}=\Bbb Q\bigoplus
\Bbb Z/p_{i}Z$ and $I_{j}=\Bbb Q\bigoplus \Bbb Z/p_{j}Z$ are not
adjacent in $G(\tau_{Spec(R)})$, but they are adjacent in
$G(\tau^{m}_{Max(R)})$. This example shows that
$G(\tau^{m}_{Max(R)})$ is not a subgraph of $G(\tau_{Spec(R)})$.

\end{eg}

The following theorem illustrates some graphical parameters.

\begin{thm}\label{t3.6} The Zariski topology-graph $G(\tau^{m}_{T})$ is
connected and $diam(G(\tau^{m}_{T}))\leq 3$. Moreover if
 $G(\tau^{m}_{T})$ containing a cycle satisfies $gr(G(\tau^{m}_{T}))\leq 4$.
\end{thm}

\begin{proof}
Use the technique of \cite[Theorem 2.10]{ah}.
\end{proof}

\begin{prop}\label{p3.7} Let $M$ be an $R$-module and let
 $\psi: Max(M)\rightarrow Max(\overline{R})$ be the natural map.
Suppose $Max(M)$ is homeomorphic to $Max(\overline{R})$ under
$\psi$. Let $N$ and $L$ be adjacent in $G(\tau^{m}_{T})$ and let
$T'=\{\overline{P:M} : P\in T\}$. Then $\overline{N:M}$ and
$\overline{L:M}$ are adjacent in $G(\tau^{m}_{T'})$. Conversely,
if $\overline{I}$ and $\overline{J}$ are adjacent in
$G(\tau^{m}_{T'})$, then $IM$ and $JM$ are adjacent in
$G(\tau^{m}_{T})$.
\end{prop}

\begin{proof}
The proof is similar to \cite[Proposition 2.11]{ah}.
\end{proof}

\begin{lem}\label{l3.8} Let $G(\tau^{m}_{T})\neq \emptyset$ and let $P\in T$.
 Then $P$ is a vertex if each of the following statements holds.
\begin {itemize}
\item [(a)]
 There exists a subset $T'$ of $T$ such that $P\in T'$, $V^{m}(\bigcap_{Q\in T'}Q)=T$,
 and $V^{m}(\bigcap_{Q\in T', Q\neq P}Q)\neq
 T$. In particular, this holds when $T$ is a finite set and every element of
 $T$ is adjacent to a semi maximal submodule of $M$.
 \item [(b)]
  For a submodule $N$ of $M$,  $N\in V(G(\tau^{m}_{T}))$ and $N\cup P\notin V(G(\tau^{m}_{T}))$.
 \end {itemize}
\end{lem}

\begin{proof}
Straightforward.
\end{proof}

\begin{defn}\label{d3.10} We define a subgraph $G_{d}(\tau^{m}_T)$ of $G(\tau^{m}_{T})$
with vertices\\ $V((G_{d}(\tau^{m}_T))$= $\{N < M: $ there exists
$L < M$ such that $V^{m}(N)\cup V^{m}(L)=T$ and
$V^{m}(N),V^{m}(L)\neq T$ and $V^{m}(N)\cup V^{m}(L)=\emptyset
\},$ where distinct vertices
 $N$ and $L$ are adjacent if and only if $V^{m}(N)\cup V^{m}(L)=T$ and
 $V^{m}(N)\cup V^{m}(L)=\emptyset$. It is clear that the degree of
  each $N\in V((G_{d}(\tau^{m}_T))$ is the number of submodules $K$ of $M$ such that
   $V^{m}(K)=V^{m}(L)$, where $N$ and $L$ are adjacent.
\end{defn}

\begin{lem}\label{l3.11}

(a) $G_{d}(\tau_{T})\neq \emptyset$ if and only if
$T=V^{m}(\Im(T))$ and $T$ is disconnected.

(b) Suppose $Spec(M)=Max(M)$ and $M$ is a Max-surjective module
and $T$ is closed. Then $G_{d}(\tau^{m}_T)=\emptyset$ if and only
if $R/(\Im(T):M)$ contains no idempotent other than $\overline{0}$
and $\overline{1}$.
\end{lem}

\begin{proof}(a) is straightforward and (b) follows from \cite [Proposition 2.9]{HR07}
 and \cite [Corollary 3.8]{lu99}.
\end{proof}

\begin{thm}\label{t3.12} $G_{d}(\tau^{m}_T)$ is a bipartite graph.
\end{thm}

\begin{proof}
Use the technique of \cite[Theorem 2.17]{ah}.
\end{proof}

\begin{cor}\label{c3.13} By Theorem \ref{t3.12}, if $G_{d}(\tau^{m}_T)$ contains a cycle,
 then $gr(G_{d}(\tau^{m}_T))=4$.
\end{cor}

\begin{eg}\label{e3.14} Set $M:= \Bbb Z/12\Bbb Z$. Then $Max(M)=\{2\Bbb Z/12\Bbb Z, 3\Bbb Z/12\Bbb Z\}$.
Set $T= Max(M)$. Clearly, $G(\tau^{m}_{T})= G_{d}(\tau^{m}_T)$ is
a bipartite graph and $\Bbb Z/(\cup_{P\in T}P:M)\cong \Bbb Z/6\Bbb
Z$
 contains idempotents other than $\overline{0}$ and $\overline{1}$.
\end{eg}

\begin{eg}\label{e3.15} Set $M:= \Bbb Z/30\Bbb Z$. Then $Max(M)=\{2\Bbb Z/30\Bbb Z, 3\Bbb Z/30\Bbb Z, 5\Bbb Z/30\Bbb Z\}$.
Set $T= Max(M)$. Clearly, $G_d(\tau_T)$ is a bipartite graph and
$\Bbb Z/(\bigcap_{P\in T}P:M)\cong \Bbb Z/30\Bbb Z$ contains
idempotents other than $\overline{0}$ and $\overline{1}$.
\end{eg}

The Example \ref{e3.15} shows that $G_{d}(\tau^{m}_T)$ is not
necessarily connected. The following proposition provides some
useful characterization for this case.

\begin{prop}\label{p3.17}
\begin {itemize}
\item [(a)] $G_{d}(\tau^{m}_T)$ with two parts $U$ and $V$ is a
complete bipartite graph if and only if for every $N, L\in U$
(resp. in $V$), $V^{m}(N)=V^{m}(L)$. \item [(b)]
$G_{d}(\tau^{m}_T)$ is connected if and only if it is a complete
bipartite graph.
\end {itemize}
\end{prop}

\begin{proof}
Use the fact that if $N, L$ are two vertices, then $d(N,L)=2$ if
and only if $V^{m}(N)=V^{m}(L)$.
\end{proof}

\section{The relationship between $G(\tau^{m}_{T})$ and $AG(M)$}

The purpose of this section is to illustrate the connection
between the Zariski topology-graph on the maximal spectrum of a
module and the annihilating-submodule graph.

\begin{thm}\label{t4.1} Let $M$ be a Max-surjective module and suppose $N$
and $L$ are adjacent in $G(\tau^{m}_{T})$. Then
 $J^{m}((N:M)M)/\Im(T)$ and $J^{m}((L:M)M)/\Im(T)$ are adjacent in $AG(M/ \Im(T))$.
\end{thm}

\begin{proof}
Assume that $N$ and $L$ are adjacent in $G(\tau^{m}_{T})$ so that
$V^{m}(N) \cup V^{m}(L)= T$. Then we have
$V^{*m}(((N:M)M)((L:M)M)))= T$.\\
 It implies that
$J^{m}(((N:M)M)((L:M)M))= \Im(T)$. Hence we have $\Im(T) \subseteq
J^{m}((N:M)M)$ and $\Im(T) \subseteq J^{m}((L:M)M)$. Now it is
enough to show that $$ J^{m}((N:M)M) J^{m}((L:M)M)=
(J^{m}((N:M)M): M)(J^{m}((L:M)M): M)M \subseteq$$
 $\Im(T)
.$\\ Since $M$ is Max-surjective, by \cite[Lemma 3.11(a)]{ak}, we
have $(J^{m}((N:M)M): M)=J^{m}(((N:M)M: M))= J^{m}((N:M))$ and
$(J^{m}((L:M)M): M)= J^{m}((L:M))$. By using \cite[Proposition
2]{lu89}, it turns out that
$$
(J^{m}((N:M)M): M)(J^{m}((L:M)M): M)M= J^{m}((N:M)) J^{m}((L:M)))M
\subseteq $$
$$
 J((N:M)(L:M))M \subseteq J^{m}((N:M)(L:M)M)= J^{m}(NL)=\Im(T)
.$$ Note that if $J^{m}((N:M)M)=\Im(T)$ or
$J^{m}((N:M)M)=J^{m}((L:M)M)$, then we have $(N:M)M \subseteq
\Im(T)$. This implies that $V^{m}(N)= T$, a contradiction. This
completes the proof.
\end{proof}

\begin{cor}\label{c4.2} Assume that the hypothesis  hold as in Theorem \ref{t4.1}.
Then $J^{m}(N)/\Im(T)$ and $J^{m}(L)/\Im(T)$ are
  adjacent in $AG(M/ \Im(T))$.
\end{cor}

\begin{proof}
By meeting the above theorem again, we see that $V^{*m}(NL)=T$,
$J^{m}(NL)=\Im(T)$, and
$$
J^{m}(N) J^{m}(L)=J^{m}(N:M)J^{m}(L:M)M\subseteq\
J^{m}(NL)=\Im(T).
$$
\end{proof}

\begin{prop}\label{p4.3} Suppose $M/ \Im(T)$ is not a vertex in $AG(M/ \Im(T))$.
 Then we have the following.
\begin {itemize}
\item [(a)]  The annihilating-submodule graph $AG(M/ \Im(T))$ is
isomorphic to a subgraph of $G(\tau^{m}_{T})$. \item [(b)] If $M$
is a Max-surjective module or $Spec(M)=Max(M)$, then\\ $AG(M/
\Im(T))=\emptyset$ if and only if $G(\tau^{m}_{T})=\emptyset$.
\item [(c)] If $R$ is an Artinian ring and $M/ \Im(T)$ is a
finitely
 generated module, then for every non-zero proper submodule $N/ \Im(T)$ of
 $M/ \Im(T)$, $N/ \Im(T)$ and $N$ are
 vertices in $AG(M/ \Im(T))$ and $G(\tau^{m}_{T})$, respectively.

\end {itemize}

\end{prop}

\begin{proof} Let us begin our proof by noting that $M/\Im(T)$ is a vertex in $AG(M/ \Im(T))$
 if and only if there exists $N < M$
containing $\Im(T)$ properly with $V^{m}(N)=T$.

(a) Let $N/ \Im(T) \in V(AG(M/ \Im(T)))$. Then there exists a
nonzero
 submodule $L/ \Im(T)$ of $M/ \Im(T)$
such that it is adjacent to $N/ \Im(T)$ ($N\neq L$, because $M/ \Im(T)$ is not a vertex).
 So we have $NL\subseteq \Im(T)$.
 Hence $V^{m}(NL)=T$. If $V^{m}(N)=T$, then $(N:M)=(\Im(T) :M)$. It follows that $M/ \Im(T)$
  is a vertex, a contradiction. Hence $N$ is a
  vertex in $G(\tau^{m}_{T})$ which is adjacent to $L$. In particular, if $M=R$ and $\Im(T)=0$,
   then $AG(R)$ is a subgraph of $G(\tau_{Max(M)})$.

(b) To see the forward implication, let $AG(M/ \Im(T))=\emptyset$.
Then $M/ \Im(T)$ is a prime $R$-module so that $\Im(T)$ is a prime
submodule of $M$. Thus we have $G(\tau^{m}_{T})=\emptyset$.
Conversely, suppose that $AG(M/ \Im(T))\neq\emptyset$. Then by
part (a), $AG(M/ \Im(T))$ is isomorphic to a subgraph of
$G(\tau^{m}_{T})$. Hence $G(\tau^{m}_{T})\neq \emptyset$, as
desired.

(c) This follows from Proposition D and part (a).
\end{proof}

We need the following theorem from \cite[Theorem 3.3]{mm92}.

\begin{thm}\label{t4.4}
Let $M$ be a finitely generated module and let $N$ be a submodule of $M$ such that $(N:M) \subseteq P$,
where $P$ is a prime ideal of $R$. Then there exists a prime submodule $K$ of $M$ such that
$N \subseteq K$ and $(K:N)= P$.
\end{thm}

\begin{thm}\label{t4.5} Suppose $N/ \Im(T)$ and $L/ \Im(T)$ are adjacent in $M/ \Im(T)$.
Then $N$ and $L$ are adjacent in $G(\tau^{m}_{T})$ if one of the
following conditions holds.
\begin{itemize}
\item [(a)] $M/ \Im(T)$ is not a vertex in $AG(M/ \Im(T))$. In
particular, this holds when $M/ \Im(T)$ is finitely generated and
contains no semi maximal submodule $S\neq \Im(T)$ with $V(S)=T$.
\item [(b)] $M/N$ and $M/L$ are Max-surjective and they contain no
semi maximal submodule $S\neq \Im(T)$ with $V(S)=T$.
\end{itemize}
\end{thm}

\begin{proof}
(a) This follows from Proposition \ref{p4.3} (a). For the proof of
the second assertion, suppose $M/ \Im(T)$ is a vertex. So
 there exists a non-zero proper submodule $N'/ \Im(T)$ of $M /\Im(T)$, where $N'<M$,
  such that $V^{m}(N')=T$. It is clear
 that $M/ \Im(T)$ has a structure
 of $R/(\Im(T):M)$-module. Let $Q$ be an arbitrary element of $T$. Then we have
 $(N'/\Im(T):M/\Im(T))\subseteq(Q:M)/(\Im(T):M)$.
  Now by Theorem \ref{t4.5}, there exists a prime submodule $K/ \Im(T)$ such that
  $N'\subseteq K$ and $(K:M)=(Q:M)$. It
   follows that $N'\subseteq\Im(T)$, a contradiction.

 (c) Clearly, $V^{m}(N)\cup V^{m}(L)=T$. Now let $V^{m}(N)=T$. Then we
  have $(N:M)\subseteq (Q:M)$ for every $Q\in T$. Since $M/N$
 is Max-surgective, there exists a maximal submodule $K$ of $M$ such that
 $N\subseteq K$ and $(K:M)=(Q:M)$. Hence $N\subseteq \Im(T)$,
 contrary to our assumption. So $V^{m}(N) \neq T$ and the proof is
completed.
 \end{proof}

 Now we put the following lemma which is needed later. Its proof is easy and is omitted.

\begin{lem}\label{l4.6}
 Let $N \leq M$ and let $dim(R)=0$. Then $rad(N)=N$ if and only if $Nil(R)M=0$
 (we recall that $Nil(R)= \sqrt{0}$ is the ideal
  consisting of the nilpotent elements of $R$).
\end{lem}

 \begin{prop}\label{p4.7} Suppose $dim(R)=0$, $Nil(R)M=0$, and $M/ \Im(T)$ is not a
  vertex in $AG(M/ \Im(T))$.
 Then the Zariski topology-graph $G(\tau^{m}_{T})$ and the annihilating-submodule graph
  $AG(M/ \Im(T))$ are isomorphic. In particular,
  when $M/ \Im(T)$ is not a prime module, $G(\tau^{m}_{T})$ has acc (resp. dcc) on vertices
   if and only if $M/ \Im(T)$ is a Noetherian (resp. an Artinian) module.
\end{prop}

\begin{proof} By Proposition \ref{p4.3} (a), It thus remains to show that $G(\tau^{m}_{T})$ is
 isomorphic with a subgraph of $AG(M/ \Im(T))$.
 Suppose $N$ and $L$ are adjacent in $G(\tau^{m}_{T})$. So we have $V^{m}(NL)=T$. From
 \cite[Theorem 2.3]{bkk}, we have $rad(N)=N$
 if and only if $J^{m}(N)=N$. Now by
 Lemma \ref{l4.6}, $J^{m}(NL)=NL$. Hence $N/ \Im(T)$
 and $L/ \Im(T)$ are adjacent in $AG(M/ \Im(T))$ as required. The second assertion
 follows from Theorem
 E.
\end{proof}

\begin{lem}\label{l4.8}
 Assume that  $M/ \Im(T)$ is not a vertex in $AG(M/ \Im(T))$. Suppose that for
 every $Q\in T\bigcap V(G(\tau^{m}_{T}))$, there exists a
 semi maximal submodule of $M$ such that it is adjacent to $Q$. Then $Max(M)\bigcap
V(G(\tau^{m}_{T}))\neq\emptyset$
   if and only if $Max(M/\Im(T))\bigcap V(AG(M/\Im(T)))\neq\emptyset$.
\end{lem}

 \begin{proof} Choose $\emph{Q}\in Max(M)$ and $\emph{Q}\in V(G(\tau^{m}_{T}))$.
 By assumption, $$V^{m}(\emph{Q})\bigcup V^{m}(\bigcap_{Q\in
 T'}Q)=T,$$
  where $T'$ is a subset of $T$. Then $\emph{Q}/\Im(T)$ and $\bigcap_{Q\in T'}Q/ \Im(T)$
   are adjacent in $AG(M/ \Im(T))$.
To see the backward implication, suppose $Q/ \Im(T)$ is a vertex
in $AG(M/ \Im(T))$. So there exists a submodule $N/\Im(T)$ of
$M/\Im(T)$ which is adjacent
 to it. It follows that $Q$ and $N$ are adjacent in $G(\tau^{m}_{T})$ and the proof is completed.

\end{proof}

 \begin{prop}\label{p4.9} Suppose that for every $mM\in T \bigcap V(G(\tau^{m}_{T}))$,
  there exists a semi maximal submodule of $M$ such that
 it is adjacent to $mM$, where $m\in Max(R)$. Then we have exactly one of the following assertions.
 \begin {itemize}
\item [(a)] There exists a non-zero submodule $K$ of $M$ with $K
\neq \Im(T)$ and $V^{m}(K)=T$. \item [(b)] There exists a maximal
ideal $m$ of $R$ such that $mM\in T \bigcap V(G(\tau^{m}_{T}))$ if
and only if $Soc(M/\Im(T))\neq 0$.
\end {itemize}

\end{prop}

\begin{proof}
An obvious way to do this is to suppose (a) doesn't hold and we
have $mM\in T \bigcap V(G(\tau^{m}_{T}))$, where $m$ is a maximal
ideal of $R$. Then by Lemma \ref{l4.8}, $mM/ \Im(T)$ is a vertex
in $AG(M/ \Im(T))$. It is clear that $M/ \Im(T)$ is not a vertex
in $AG(M/ \Im(T))$. Thus by Proposition G, $Soc(M/\Im(T))\neq 0$.
Conversely, let $Soc(M/\Im(T))\neq 0$. Then
 by Proposition G, $mM/\Im(T) \in V(AG(M/\Im(T)))$, where $m$ is a maximal ideal of $R$. Hence
  by Lemma \ref{l4.8}, $mM\in V(G(\tau^{m}_{T}))$ and the proof is completed.

\end{proof}

\begin{thm}\label{t4.10}
Assume that $M/ \Im(T)$ is a faithful module which is not a vertex
in $AG(M/ \Im(T))$. Then the following assertions are equivalent.
\begin {itemize}
\item [(a)]
 $G(\tau^{m}_{T})$ is a finite graph.
\item [(b)] $AG(M/ \Im(T))$ is a finite graph \item [(c)]
$M/\Im(T)$ has finitely many submodules. Moreover,
$G(\tau^{m}_{T})$ has n $(n\geq 1)$ vertices if and only if
$M/\Im(T)$ has only $n$ nonzero proper submodules.
\end {itemize}

\end{thm}

\begin{proof}

$(a)\Rightarrow (b)$ This follows from Proposition \ref{p4.3} (a).

 $(b) \Rightarrow (c)$ Suppose $AG(M/ \Im(T))$ is a finite graph with $n$ vertices ($n\geq 1$).
 By Theorem E, $M/ \Im(T)$ has finite length.
  Now the claim follows from Proposition D.

$(c) \Rightarrow (a)$ $M/ \Im(T)$ has finite length so that
$dim(R)=0$. Also, it follows from the hypothesis that $Nil(R)=0$.
Therefore, $G(\tau^{m}_{T})$
 is a finite graph by Proposition \ref{p4.7}. The second assertion follows easily from the
  above arguments.
\end{proof}

The following example shows that
  the hypotheses "$M/ \Im(T)$ is a faithful module" is
  needed in the above theorem.

\begin{eg}\label{e4.11} Consider Example \ref{e3.5} (case (1)). When $|T|\geq 2$ and
$$T\subseteq \{p_{1}\Bbb Z\bigoplus \Bbb Z, \ldots
,p_{n}\Bbb Z\bigoplus \Bbb Z\}.$$ Then every element $T$ is a
vertex and $G(\tau^{m}_{T})$ is an infinite graph because
$p_{i}^{k}\Bbb Z\bigoplus \Bbb Z$ is a vertex for every positive
integer $k$ and $1\leq i\leq n$. Now we have
$$R/\Im(T)=\Bbb Z\bigoplus \Bbb Z/(p_{1}p_{2}\ldots
p_{k}\Bbb Z)\bigoplus \Bbb Z.$$ It follows that $AG(R/\Im(T))$ is
a finite graph.
\end{eg}

 We end this work with the following question:

\begin{que}\label{q4.12}
Let $G(\tau^{m}_{T})\neq \emptyset$, where $T$ be an infinite
subset of $Max(M)$. Is $T\bigcap V(G(\tau^{m}_{T}))\neq
\emptyset?$
\end{que}

\end{document}